\documentclass[
11pt, % The default document font size, options: 10pt, 11pt, 12pt
english, % ngerman for German
singlespacing, % Single line spacing, alternatives: onehalfspacing or doublespacing
parskip, % Uncomment to add space between paragraphs
headsepline, % Uncomment to get a line under the header
]{amsart} % The class file specifying the document structure

\usepackage[utf8]{inputenc} % Required for inputting international characters
\usepackage[T1]{fontenc} % Output font encoding for international characters
\usepackage{amsthm} % Allows for theorem styles, particularly allows for a different style for definitions
\usepackage{mathpazo} % Use the Palatino font by default
\usepackage{mathtools} % for the hooked arrow
\usepackage{amssymb} % for the surjection arrow
\usepackage{tikz} % for diagrams
\usepackage{caption}
\usepackage{float}
\usepackage{cases}
\usepackage{oubraces}

\usepackage{enumerate}
\usepackage{pb-diagram}
\usepackage{microtype}
\usepackage{amsmath}
\usepackage{amssymb,pb-diagram}
\usepackage[all]{xy}
\usepackage{graphicx} 
\usepackage{mathabx}
\usepackage{epsfig}
\usepackage{amsfonts}
\usepackage{amsthm}
\usepackage{color}
\usepackage{latexsym}
\usepackage[capitalise]{cleveref}

\usepackage[backend=bibtex,style=alphabetic,natbib=true]{biblatex} % Use the bibtex backend with the authoryear citation style (which resembles APA)

\usepackage[autostyle=true]{csquotes} % Required to generate language-dependent quotes in the bibliography
\newtheorem{theorem}{Theorem}[section]
\newtheorem{citing}[theorem]{Citation}
\newtheorem{lemma}[theorem]{Lemma}

\newtheorem{corollary}[theorem]{Corollary}
\newtheorem{conjecture}[theorem]{Conjecture}

\theoremstyle{definition}
\newtheorem{definition}[theorem]{Definition}
\newtheorem{example}[theorem]{Example}

\newcommand{\TF}{\textbf{F}}
\newcommand{\FN}[1]{$\textbf{F}_{#1}$}
\newcommand{\drawcaret}[2]{\draw (#1,#2) -- (#1-0.5,#2-1); \draw (#1,#2) -- (#1+0.5,#2-1)}

\begin{document}
\title[Tree Pairs for Algebraic Bieri-Strebel Groups]{Tree Pairs for Algebraic Bieri-Strebel Groups}

\author[L.~Molyneux]{Lewis Molyneux}
\address{
Royal Holloway, University of London\\ 
Department of Mathematics, McCrea Building, TW20~0EX Egham, UK}

\begin{abstract}
We reintroduce a previously discovered method for constructing tree pair representations for Algebraic Bieri-Strebel groups, as well as demonstrate a class of higher order groups that cannot have a tree pair representation. In doing so, we demonstrate that there is no maximum degree such that for all polynomials of higher degree, the associated Algebraic Bieri Strebel group must have a tree-pair representation.
\end{abstract}

\maketitle

\section{Introduction}

Tree pair representations. are of distinct importance in the study of Thompson-like groups, particularly when it comes to research into finiteness properties. Even when just considering Thompson's group $\TF$, the method of expressing elements as equivalence classes of ordered tree pairs is integral when determing that $\TF$ has the $F_\infty$ property \cite{Brown} \cite{ZarAccount}, calculating an infinite and finite presentation \cite{CFP} \cite{Fbook}, calculating the BNSR invariant \cite{Witzel}, and even demonstrating that $\TF$ acts without fixed point on a $CAT(0)$ cube complex \cite{Farley}. This approach has been similarly fruitful with other members of the family of Thompson-like groups. Not only is it possible to construct such a representation for such diverse groups as Stein's group \cite{Stein}, the Lodha-Moore groups \cite{LMbook}, and even Brin's group 2V \cite{2V}, and has been useful for calculating properties among many of these groups \cite{F23}, \cite{Zaremsky}, \cite{Wladis}.

Bieri-Strebel groups are a class of groups that arise from a generalisation of Thompson's group $\TF$. While Thompson's group is typically understood as the group of piecewise linear, orientation preserving homeomorphisms of $[0,1]$, with slopes in $\langle 2 \rangle$ and breakpoints in $\mathbb{Z}[\frac{1}{2}] =\{ \frac{a}{2^b} | a,b \in \mathbb{N}\}$, Bieri and Strebel's construction allows us to generalise many aspects of this initial group. Given an interval of the real line $I$, a multiplicative group of real numbers $P$ and $A$ a $\mathbb{Z}P$ submodule of the additive group of real numbers, we can define $G(I,A,P)$ as the group of piecewise linear, orientation preserving homeomorphisms of $I$ with slopes in $P$ and breakpoints in $A$ \cite{BSgroups}. This is an extremely large class of groups, and aside from work determining general properties for Bieri-Strebel groups in general, the bulk of research into this class of groups has been focused on the subclass of groups expressible as $G([0,1], \mathbb{Z}[\beta], \langle \beta \rangle)$, where $\beta$ is some algebraic number in $\mathbb{R}$. Note that this class includes $\TF$ as $G([0,1], \mathbb{Z}[\frac{1}{2}], \langle 2 \rangle)$, but also each of the Brown-Thompson groups \FN{n} as $G([0,1], \mathbb{Z}[\frac{1}{n}], \langle n \rangle)$.

Research into the class of algebraic Bieri-Strebel groups began in earnest with Cleary in \cite{poset} and \cite{Cleary}, but it was not until Burillo, Nucinkis and Reeves \cite{Ftau} that a tree-pair representation was discovered for an algebraic Bieri-Strebel group with a non-linear associated algebraic number. However, this paper was focused on the group \FN{\tau}, and did not provide a general tree-pair representation for any class of algebraic Bieri-Strebel groups. Winstone \cite{Winstone} was able to generalise this work, developing tree pairs for a large class of algebraic Bieri-Strebel groups with quadratic algebraic numbers. However, Winstone also discovered a class of algebraic Bieri-Strebel groups that cannot have a tree-pair representation. This is highly unusual for a class of Thompson-like groups, and invalidates many methods mathematicians would often use to research such groups. Thompson-like groups share many similar properties, particularly finiteness properties. For example, for many Thompson-like groups that have a known BNSR invariant, including Thompson's group $\TF$ \cite{SigmaF}, the Brown-Thompson groups \FN{n} \cite{Fn} \cite{Zaremsky}, Stein's group \FN{2,3} \cite{F23} and \FN{\tau} \cite{Mine}, the BNSR invariant has a very similar structure, Similarly, we generally expect Thompson-like groups to have the $F_\infty$ property. In many cases, these properties can be derived from the tree pair representation. In the case of Thompson-like groups without a tree pair representation, very little can be easily deduced about their properties. In the specific case of Bieri-Strebel groups without tree-pair representations, all that is really known is results regarding Bieri-Strebel groups in general, such as Tanner's finite generation result \cite{Tanner}. It is possible that these groups have wildly different properties compared to other Thompson-like groups. 

This paper expands the class of algebraic Bieri-Strebel groups that cannot have a tree pair representation. Specifically, it proves the following theorem.

\begin{theorem}
The algebraic Bieri-Strebel group $G([0.1], \mathbb{Z}[\beta], \langle \beta \rangle)$, where $\beta$ is the root of the polynomial $ax^{2n}+bx^n-1$ with $0 < \beta <1$, cannot have a tree-pair representation.
\end{theorem}

This class contains algebraic Bieri-Strebel groups with associated algebraic number of arbitrarily high degree, making it distinct from Winstone's class of exclusively quadratic Bieri-Strebel groups.

\section{Thompson's group}

\begin{definition}\label{Fdef}
Thompson's group $\TF$  is the group of piecewise linear, orientation preserving homeomorphisms of the unit interval $I = [0,1]$ such that (\cite{Fbook}, 1.1.1)
\begin{itemize}
\item the slopes of each linear segment are in the multiplicative group $\langle 2 \rangle$
\item the breakpoints, or points between the slopes, fall in $I \bigcap \mathbb{Z}[\frac{1}{2}] = \{\frac{a}{2^b} \in I| a, b \in \mathbb{N}\}$
\end{itemize}
\end{definition}

This group, along with the related groups \textbf{T} and \textbf{V}, were originally studied by Richard Thompson in 1965. $\TF$ has been shown to have the $F_\infty$ property, and is therefore finitely presented. The finite presentation can be seen in (\cite{Fbook}, 2.4.1), but is not useful for our purposes. It is more common to work with the infinite presentation(\cite{Fbook}, 2.1.1)

\begin{equation}\label{Fpres}
\textbf{F} = \langle x_0, x_1, x_2, \dots | x_j x_i = x_i x_{j+1}, j>i \rangle
\end{equation}

Many properties, including this infinite presentation, are derived from a property of $\TF$ known as the tree pair presentation

\subsection{Tree Pairs for \TF}\label{treepair}

Considering each element of $\TF$ as a function mapping $I$ to $I$, we can see that each such function, constrained by the definition \ref{Fdef}, can be defined solely by the location of its breakpoints, as the function is just the linear segments that connect each breakpoint. Each breakpoint of the function consists of a value on the domain interval, and a value on the codomain interval. Elements of $\TF$ are orientation preserving homeomorphisms of $I = [0,1]$, so this requires that for any $f \in TF$, $f(0) = (0) and f(1) = 1$. For a given element of $\TF$, we can take the value of each breakpoint on the domain and form a set $\{x_0,\dots,x_n\}$, where $x_0 = 0$, $x_1 = 1$, and $j>i \Rightarrow x_j > x_i$. This is what Cannon, Floyd and Parry refer to as a partition of the interval \cite{CFP}, section 2).

Given that all our breakpoints sit in $I \bigcap \mathbb{Z}[\frac{1}{2}]$, we can calculate the lengths of the intervals between them and express them in the form $\frac{a}{2^b}$, $a, b \in \mathbb{N}$. We can then refine this interval partition by introducing redundant breakpoints, or breakpoints where the function doesn't change gradient, allowing us to rewrite any interval of length $\frac{a}{2^b}$ as $a$ intervals of length $\frac{1}{2^b}$. This creates what Cannon, Floyd and Parry refer to as a dyadic partition. We may perform this refinement on both the domain interval partition $\{x_{d,0}, x_{d,1}, \dots, x_{d,n}\}$, and the codomain interval partition $\{x_{c,0}, x_{c,1}, \dots, x_{c,n}\}$. We can then consider the element $f \in \TF$ as the function that maps the domain interval $[x_{d,i}, x_{d,i+1}]$ to the codomain interval $[x_{c,i}, x_{c,i+1}]$ for $0 \leq i \leq n-1$. This produces the interval partition representation for elements of $\TF$.

\begin{center}
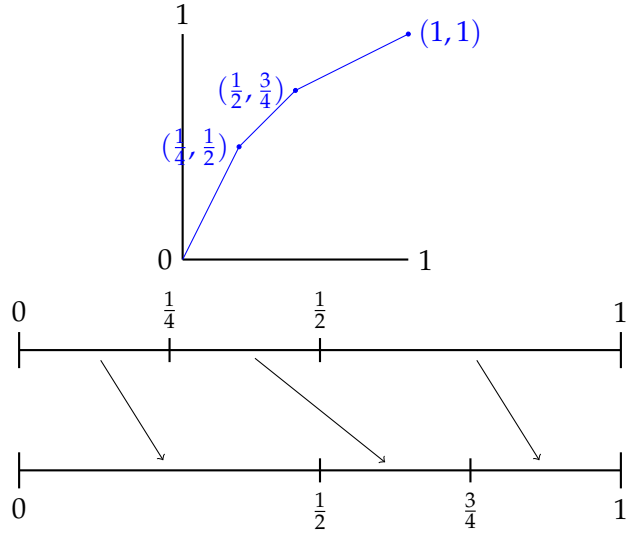
\begin{figure}[H]

\begin{tikzpicture}[scale=3]
\draw[thick](0,0) -- (0,1);
\draw[thick](0,0) -- (1,0);
\draw[blue](0,0) -- (0.25,0.5);
\draw[blue](0.25,0.5) -- (0.5,0.75);
\draw[blue](0.5,0.75) -- (1,1);
\filldraw[blue] (0.25,0.5) circle (0.25pt) node[anchor=east]{$(\frac{1}{4},\frac{1}{2})$};
\filldraw[blue] (0.5,0.75) circle (0.25pt) node[anchor=east]{$(\frac{1}{2},\frac{3}{4})$};
\filldraw[blue] (1,1) circle (0.25pt) node[anchor=west]{$(1,1)$};
\filldraw (0,0) circle (0pt) node[anchor=east]{$0$};
\filldraw (1,0) circle (0pt) node[anchor=west]{$1$};
\filldraw (0,1) circle (0pt) node[anchor=south]{$1$};
\end{tikzpicture}
\hspace{10mm}
\begin{tikzpicture}[scale=8]
    \draw[thick] (1,0) -- (0,0);
    \draw[thick] (0,-0.2) -- (1, -0.2);
    \draw[thick] (0, 0.03) -- (0, -0.03);
    \draw[thick] (1,-0.03) -- (1, 0.03);
    \draw[thick] (0, -0.23) -- (0,-0.17);
    \draw[thick] (1, -0.23) -- (1, -0.17);
    \draw[thick] (0.25, 0.02) -- (0.25,-0.02);
    \draw[thick] (0.5, 0.02) -- (0.5, -0.02);
    \draw[thick] (0.5, -0.22) -- (0.5, -0.18);
    \draw[thick] (0.75, -0.22) -- (0.75, -0.18);
    \filldraw (0,0.03) circle (0pt) node[anchor=south]{$0$};
    \filldraw (1,0.03) circle (0pt) node[anchor=south]{$1$};
    \filldraw (0,-0.23) circle (0pt) node[anchor=north]{$0$};
    \filldraw (1,-0.23) circle (0pt) node[anchor=north]{$1$};
    \filldraw (0.25,0.02) circle (0pt) node[anchor=south]{$\frac{1}{4}$};
    \filldraw (0.5,0.02) circle (0pt) node[anchor=south]{$\frac{1}{2}$};
    \filldraw (0.5,-0.22) circle (0pt) node[anchor=north]{$\frac{1}{2}$};
    \filldraw (0.75,-0.22) circle (0pt) node[anchor=north]{$\frac{3}{4}$};
    \node (A) at (0.125,0) {};
    \node (B) at (0.375,0) {};
    \node (C) at (0.75,0) {};
    \node (D) at (0.25,-0.2) {};
    \node (E) at (0.625,-0.2) {};
    \node (F) at (0.875,-0.2) {};
    \draw[->](A)--(D);
    \draw[->](B)--(E);
    \draw[->](C)--(F);

\end{tikzpicture}

\caption{The same element of $\TF$, expressed as both a function $f:I \rightarrow I$ and as a pair of partitions.}\label{x01}
\end{figure}
\end{center}

We can convert an interval partition representation into a tree pair representation by taking each dyadic partition of the interval, and converting it into a rooted binary tree in the following way

\begin{itemize}
\item each node of the tree represents a standard dyadic interval (ie: an interval of length $\frac{1}{2^a}$, $a \in \mathbb{Z}_{\geq 0}$).
\item the root node of a tree represents the interval $I$.
\item For a node with corresponding interval $[\frac{a}{2^b}, \frac{a+1}{2^b}]$, the two nodes directly beneath it represent $[\frac{2a}{2^b},\frac{2a+1}{2^{b+1}}]$ (on the left) and $[\frac{2a+1}{2^{b+1}},\frac{a+1}{2^b}]$ (on the right).
\end{itemize} 

Thus a leaf (a node with no nodes beneath it) of depth $n$ represents an interval of length $\frac{1}{2^n}$ in the interval partition, where we consider the root to have depth $0$, and any other node has depth equal to the number of edges on the unique path between it and the root. 

Just as two interval partitions can represent a homeomorphism, two rooted binary trees can represent an element of $\TF$ by considering that the interval represented by the $n$th leaf (counting left to right) of the leftmost tree is mapped linearly to the $n$th leaf of the rightmost tree. 

\begin{center}
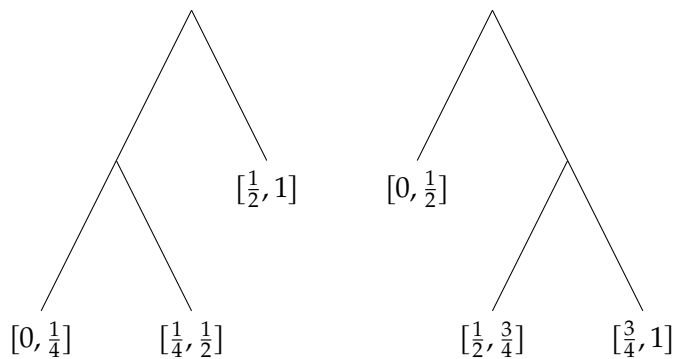
\begin{figure}[H]
\centering
\begin{tikzpicture}[scale=2]
\drawcaret{0}{0};
\drawcaret{2}{0};
\drawcaret{-0.5}{-1};
\drawcaret{2.5}{-1};
\filldraw (-1,-2) circle (0pt) node[anchor=north]{$[0,\frac{1}{4}]$};
\filldraw (0,-2) circle (0pt) node[anchor=north]{$[\frac{1}{4},\frac{1}{2}]$};
\filldraw (0.5,-1) circle (0pt) node[anchor=north]{$[\frac{1}{2},1]$};
\filldraw (1.5,-1) circle (0pt) node[anchor=north]{$[0,\frac{1}{2}]$};
\filldraw (2,-2) circle (0pt) node[anchor=north]{$[\frac{1}{2},\frac{3}{4}]$};
\filldraw (3,-2) circle (0pt) node[anchor=north]{$[\frac{3}{4},1]$};

\end{tikzpicture}
\caption{The element from \ref{x01}, now presented as a tree-pair. Note how the nodes in the trees are positioned similarly to the breakpoints in the partitions.}
\end{figure}
\end{center}

We can turn this representation of individual group elements into a representation for the whole group with two important steps. The first is to impose an equivalence relation on the set of tree pairs such that there is exactly one equivalence class for each element of $\TF$, and the second is to define a binary relation between equivalence classes of tree pairs that is equivalent to function composition as the binary relation in $\TF$.

The first step is to recognise that the insertion of a redundant breakpoint in the interval partition representation of an element is reflected by the addition of a pair of redundant carets into the tree pair representation. Just as a redundant breakpoint splits a dyadic interval and the interval it's mapped to in half, a pair of redundant carets splits a leaf on the leftmost tree and the corresponding leaf on the rightmost tree by adding a caret to each of them. (\cite{Fbook}, page 14). As redundant breakpoints don't change the underlying function, then neither do pairs of redundant carets. As such, we can form an equivalance relation on the set of tree pairs such that two tree pairs are related if one is obtainable from another by any finite sequence of additions or removals of pairs of redundant carets. 

In each equivalence class under this relation, we can produce at least one "reduced" tree pair diagram, that contains no redundant carets. Indeed, this reduced diagram is unique within that equivalence class as shown in (\cite{Fbook}, 2.3.5).

\begin{center}
\begin{figure}[H]
\centering
\begin{tikzpicture}[scale=1.2]
\drawcaret{0}{0};
\drawcaret{-0.5}{-1};
\draw (-1,-2) -- (-1.4,-3);
\draw (-1,-2) -- (-0.6,-3);
\drawcaret{2}{0};
\drawcaret{2.5}{-1};
\drawcaret{3}{-2};
\draw[red] (0,-2) -- (-0.4,-3);
\draw[red] (0,-2) -- (0.4,-3);
\draw[red] (2.5,-3) -- (2,-4);
\draw[red] (2.5,-3) -- (3,-4);

\drawcaret{6}{0};
\drawcaret{5.5}{-1};
\draw (5,-2) -- (4.6,-3);
\draw (5,-2) -- (5.4,-3);
\drawcaret{8}{0};
\drawcaret{8.5}{-1};
\drawcaret{9}{-2};

\filldraw (-1.4,-3) circle (0pt) node[anchor=north]{$1$};
\filldraw (-0.6,-3) circle (0pt) node[anchor=north east]{$2$};
\filldraw (-0.4,-3) circle (0pt) node[anchor=north west]{$3$};
\filldraw (0.4,-3) circle (0pt) node[anchor=north]{$4$};
\filldraw (0.5,-1) circle (0pt) node[anchor=north]{$5$};

\filldraw (1.5,-1) circle (0pt) node[anchor=north]{$1$};
\filldraw (2,-2) circle (0pt) node[anchor=north]{$2$};
\filldraw (2,-4) circle (0pt) node[anchor=north]{$3$};
\filldraw (3,-4) circle (0pt) node[anchor=north]{$4$};
\filldraw (3.5,-3) circle (0pt) node[anchor=north]{$5$};
\end{tikzpicture}
\caption{A tree-pair diagram with a redundant caret, highlighted in red, and the equivalent reduced diagram.}
\end{figure}
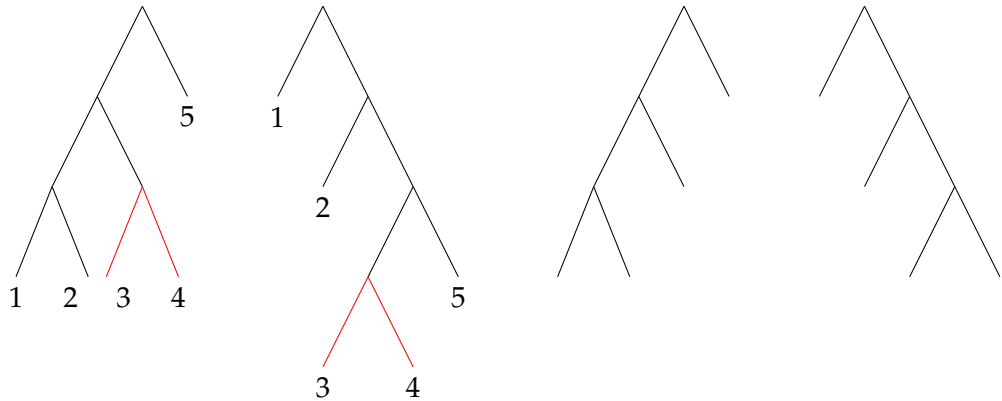
\end{center}

We utilise this equivalence class in building our binary operation. Considering two elements of $\TF$ represented as two reduced tree pairs, we can find the tree pair representing the composition of these elements by adding redundant carets to each tree pair until the rightmost tree of the first tree pair is identical to the leftmost tree of the second tree pair. The tree pair formed of the leftmost tree of the first tree pair and the rightmost tree of the second tree pair is in the equivalence class of tree pairs representing the composition of our original two elements, though it is not necessarily the reduced representative of that equivalence class (\cite{Fbook}, section 1.3)

\begin{center}
\begin{figure}[H]
\centering
\begin{tikzpicture}[scale=0.8]
\drawcaret{0}{0};
\drawcaret{-0.5}{-1};
\drawcaret{-1}{-2};
\drawcaret{2}{0};
\draw (1.5,-1) -- (1.1,-2);
\draw (1.5,-1) -- (1.9,-2);
\draw (2.5,-1) -- (2.1,-2);
\draw (2.5,-1) -- (2.9,-2);
\drawcaret{5}{0};
\draw (4.5,-1) -- (4.1, -2);
\draw (4.5,-1) -- (4.9, -2);
\drawcaret{4.9}{-2};
\drawcaret{7}{0};
\drawcaret{7.5}{-1};
\drawcaret{8}{-2};

\filldraw (0.75,-3.2) circle (0pt) node[anchor=north]{$f$};
\filldraw (6.25,-3.2) circle (0pt) node[anchor=north]{$f'$};

\drawcaret{0}{-4};
\drawcaret{-0.5}{-5};
\drawcaret{-1}{-6};
\draw[red] (-0.5,-7) -- (-1,-8);
\draw[red] (-0.5,-7) -- (0,-8);
\drawcaret{2}{-4};
\draw (1.5,-5) -- (1.1,-6);
\draw (1.5,-5) -- (1.9,-6);
\draw[red] (1.9,-6) -- (1.4, -7);
\draw[red] (1.9,-6) -- (2.4, -7);
\draw (2.5,-5) -- (2.1,-6);
\draw (2.5,-5) -- (2.9,-6);
\drawcaret{5}{-4};
\draw[red] (5.5,-5) -- (5.1, -6);
\draw[red] (5.5,-5) -- (5.9, -6);
\draw (4.5,-5) -- (4.1, -6);
\draw (4.5,-5) -- (4.9, -6);
\drawcaret{4.9}{-6};
\drawcaret{7}{-4};
\drawcaret{7.5}{-5};
\drawcaret{8}{-6};
\draw[red] (8.5,-7) -- (9,-8);
\draw[red] (8.5,-7) -- (8, -8);

\drawcaret{2}{-9};
\drawcaret{1.5}{-10};
\drawcaret{1}{-11};
\draw[red] (1.5,-12) -- (1,-13);
\draw[red] (1.5,-12) -- (2,-13);
\drawcaret{5}{-9};
\drawcaret{5.5}{-10};
\drawcaret{6}{-11};
\draw[red] (6.5,-12) -- (7,-13);
\draw[red] (6.5,-12) -- (6, -13);

\filldraw (3.5,-13.2) circle (0pt) node[anchor=north]{$f*f'$};

\end{tikzpicture}
\caption{Demonstration of composition of tree-pairs.}
\end{figure}
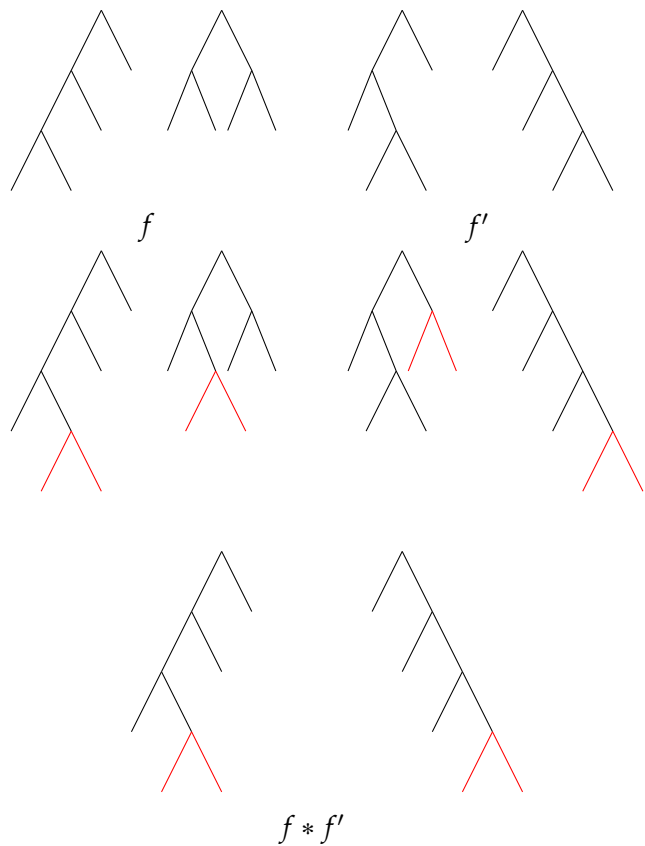
\end{center}

\subsection{\FN{n}.}\label{fn}
While there are many possible ways to generalise Thompson's group $\TF$, one that provides a useful middle ground between Thompson's group and the Bieri-Strebel groups is the family of Brown-Thompson groups \FN{n}. Many concepts from Brown-Thompson groups will carry over to Algebraic Bieri-Strebel groups and many of the concepts we have discussed and shall discuss for $\TF$ carry over to the \FN{n} groups. It is therefore prudent to define and discuss the Brown-Thompson groups before proceeding to the Bieri-Strebel Groups.

\begin{definition}
For $n \in \mathbb{N}, n \geq 2$, the Brown-Thompson group \FN{n} is the group of piecewise-linear, orientation preserving homeomorphisms of the unit interval $[0,1]$ (commonly written as $I$) such that:

\begin{itemize}
\item All gradients are in $\langle n \rangle$.
\item There are finitely many breakpoints between the slopes.
\item All breakpoints fall in $I \bigcap \mathbb{Z}[\frac{1}{n}]$.
\end{itemize}

\end{definition}

Evidently, the Brown-Thompson groups are a natural generalisation of Thompson's group. It can be seen that Thompson's group $\TF$ is the same as the Brown-Thompson group \FN{2}. The Brown-Thompson groups have a similar presentation to $\TF$.

\begin{equation}\label{FNpres}
\textbf{F}_n = \langle x_0, x_1, ..., | x_j x_i = x_i x_{j+n-1} \textsf{ for } j>i \rangle
\end{equation}

Similar to $\TF$, the Brown-Thompson groups \FN{n} can be represented with both partition pairs and tree-pair diagrams. The primary difference between tree-pairs representing different Brown-Thompson groups is the shape of the carets. In $\TF$, each time we subdivide an interval, we split it into $2$ equal pieces. This concept generalises to \FN{n}, where we subdivide each interval into $N$ equal pieces. This means that, in the tree pair diagrams for \FN{n}, each caret has $n$ legs. That is to say that each node either has $0$ direct descendants (and is therefore a leaf), or $n$ direct descendants.
\par
Other than this difference, tree-pairs function identically in \FN{n} to how they function in $\TF$. Redundant carets may be detected in the same way. We still form the same equivalence classes of tree-pairs, and tree-pair composition functions in the exact same way. As such, many properties for $\TF$ derived from tree-pair diagrams can be analogised to \FN{n}, even if those properties are not the same. The presentation in \ref{FNpres} is a good example of this, as it is similar to the presentation for $\TF$, but changes as $n$ changes. 

\section{Bieri-Strebel Groups}

Thompson's groups has been generalised in multiple different ways, with some notable examples being Stein's groups \cite{Stein} and Brin's group \cite{Brin2}. Our chosen method of generalisation was created by Bieri and Strebel, and allows for flexibility in the interval for the homeomorphisms, the slopes of the homeomorphisms, and where the breakpoints fall. 

\begin{definition}\label{BSgroups}
(\cite{BSgroups}, page ii) For an interval of the real numbers $I$, a multiplicative subgroup of the group of positive real numbers $P$ and a $\mathbb{Z}[P]$ submodule of the real numbers $A$, we define the Bieri-Strebel group $G(I,A,P)$ as the group of piecewise-linear, orientation preserving homeomorphisms of $I$ such that:

\begin{itemize}
\item All gradients are in $P$.
\item There are finitely many breakpoints between the slopes.
\item All breakpoints fall in $I \bigcap A$.
\end{itemize}
\end{definition}

Thompson's group $\TF$ is the group $G([0,1],\mathbb{Z}[\frac{1}{2}], \langle 2 \rangle)$. We can also define Brown's generalisation of Thompson's group (\cite{Brown}) \FN{n} as the group $G([0,1],\mathbb{Z}[\frac{1}{n}], \langle n \rangle)$.
\subsection{Algebraic Numbers and Subdivision Polynomials} \label{subdivide}

As Bieri-Strebel groups are very general in their basic definition, we consider a smaller class of groups. We call this class the class of Algebraic Bieri-Strebel groups.

\begin{definition}\label{algdef}
For a positive algebraic number $\beta$, the algebraic Bieri-Strebel group \FN{\beta} is the Bieri-Strebel group $G([0,1], \mathbb{Z}[\beta], \langle \beta \rangle)$.
\end{definition} 

Algebraic Bieri-Strebel groups are particularly useful to study as groups of partition pairs, as discussed in \ref{treepair}. This is due to algebraic numbers association with polynomials, which we may use to build partitions.

\begin{definition}
A subdivision polynomial is a polynomial of the form $a_n x^n + a_{n-1} x^{n-1} +...+ a_1 x -1$, where $a_i$ are all in $\mathbb{Z}_{\geq 0}$.
\end{definition}

\begin{lemma}
(\cite{Winstone}, Lemma 2.2.1) Any nontrivial subdivision polynomial $P(x)$ has exactly one real root greater than $0$, and that root lies between $0$ and $1$.
\end{lemma}
\begin{proof}
We can clearly see that $P(0)=-1$ and that $P(1)>0$ as long as we don't have $a_i = 0$ $\forall i$ or that exactly one $a_i$ equals $1$ and all else equal $0$ (these cases are trivial and can be disregarded). As $P(x): \mathbb{R} \rightarrow \mathbb{R}$ is a continuous function, there must be some value $0 <\beta < 1$ such that $P(\beta) = 0$.
\par
To show that $\beta$ is unique, we will suppose there exists $\beta' \neq \beta$. Without loss of generality, assume $\beta < \beta'$. As $a_i$ are all non-negative (excluding again the trivial case where $a_i =0$ for all $i$), we know that $P(\beta) < P(\beta')$. As such, they cannot be equal and $P(\beta') \neq 0$.
\end{proof}

Please note that in \cite{Winstone}, Winstone expresses his subdivision polynomials in the form $x^n-a_1 x^{n-1}+...+a_{n-1} x + a_n$. This has uses in his case, but is not our prefered method of expressing subdivision polynomials. The primary difference this creates is that our subdivision polynomials have a unique root between $0$ and $1$, while Winstone's polynomials have a unique root greater than $0$.

The reason we call polynomials of this form "subdivision polynomials" is that they form a method of subdividing an interval into $\sum_{i=1}^n a_i$ intervals, each with a length in $\langle \beta \rangle$. This is extremely useful for assembling partition pair representations of elements of \FN{\beta}. This can be achieved by writing $P(\beta)=0$, where $\beta$ is the positive real root of our subdivision polynomial $P$. As each subdivision polynomial has a $-1$ constant term, we can rewrite this equation as $a_n \beta^n + a_{n-1} \beta^{n-1} +... + a_1 \beta = 1$. As we will generally choose $1$ as the length of our interval when constructing Algebraic Bieri-Strebel Groups, we can interpret this equation as the sum of the length of these segments (which all have length expressible in the form $\beta^k$ for some $k \in \mathbb{N}$) is equal to $1$. Thus, we can divide the interval into $\sum_{k=1}^n a_k$ segments, comprised of $a_n$ segments of length $\beta^n$, $a_{n-1}$ segments of length $\beta^{n-1}$ and so on. We can find such a partition for any $\beta$ that is the positive real root of a subdivision polynomial.

\begin{center}
\begin{figure}[H]
\centering
\begin{tikzpicture}[scale=8]
\draw[thick] (0,0) -- (1,0);
\draw[thick] (0,-0.1) -- (0,0.1);
\draw (0.2, -0.05) -- (0.2, 0.05);
\draw (0.6, -0.05) -- (0.6, 0.05);
\draw[thick] (1,-0.1) -- (1, 0.1);
\filldraw (0.1,0) circle (0pt) node[anchor=south]{$\beta^2$};
\filldraw (0.4,0) circle (0pt) node[anchor=south]{$\beta$};
\filldraw (0.8,0) circle (0pt) node[anchor=south]{$\beta$};
\filldraw (0,0.1) circle (0pt) node[anchor=south]{$0$};
\filldraw (1,0.1) circle (0pt) node[anchor=south]{$1$};
\end{tikzpicture}
\caption{A partition of the unit interval based on the subdivision polynomial $x^2 + 2x -1$}
\end{figure}
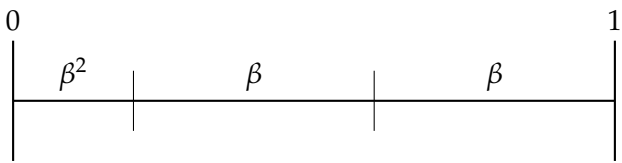
\end{center}

This method of partition can be applied to any of the subintervals created by the previous partition in a manner similar to the repeated bisection seen when forming partition pairs for $\TF$. As long as the same polynomial partition is consistently used, all created subintervals will have length in $\langle \beta \rangle$. This is extremely useful when forming elements of \FN{n} as the gradient of the piece of any element that maps $\beta^i$ to $\beta^j$ will be $\beta^{j-i}$, which will always be in $\langle \beta \rangle$, the slope group for \FN{\beta}.
\par
Applying the methods of subdivision polynomials to our previous examples allows us to see Algebraic Bieri Strebel groups as a natural generalisation of $\TF$ and \FN{n}. As partitions in $\TF$ split an interval into $2$ equal parts, and partitions in \FN{n} split an interval into $n$ equal parts, we can assign them the subdivision polynomials $2x-1$ and $nx-1$ respectively. Thus we can see \FN{n} as the "linear" Bieri-Strebel groups, that is to say the algebraic Bieri-Strebel groups with linear subdivision polynomials.

\subsection{Winstone's Tree Pairs}\label{wintree}

An important result for our understanding of Algebraic Bieri-Strebel groups, and those with quadratic subdivision polynomials in particular, was the development of tree-pair presentations for these groups. Tree-pairs were originally introduced for the group with subdivision polynomial $x^2+x-1$, commonly written as \FN{\tau}, by Burillo, Nucinkis and Reeves in \cite{Ftau}. This paper has three major conclusions for the development of tree-pair representations for algebraic Bieri-Strebel groups. The first is that the different lengths of intervals, all of which are of the form $\beta^k$ for the group \FN{\beta}, correspond to different depths on the tree, and as partitions in Bieri-Strebel groups create intervals of different lengths, the carets representing those partitions should have legs reaching to different depths.
\par
The Second realisation follows on from the first. As partitions are no longer creating multiple intervals of the same length, the order that those intervals appear in is important. To represent this in tree pairs, we have to make use of multiple different carets, which order their legs of different length in different ways.

\begin{center}
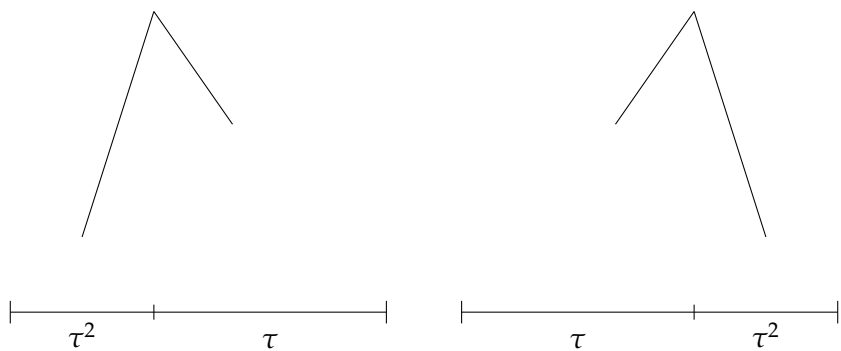
\begin{figure}[H]
\centering
\begin{tikzpicture}[scale=5]
\draw (-1.1,0)--(-0.1,0);
\draw (0.1,0)--(1.1,0);
\draw(0.1,0.03)--(0.1,-0.03);
\draw(1.1,0.03)--(1.1,-0.03);
\draw(-0.1,0.03)--(-0.1,-0.03);
\draw(-1.1,0.03)--(-1.1,-0.03);
\draw(0.7182,0.02)--(0.7182,-0.02);
\draw(-0.7182,0.02)--(-0.7182,-0.02);
\filldraw (0.4091,-0.03) circle (0pt) node[anchor=north]{$\tau$};
\filldraw (-0.4091,-0.03) circle (0pt) node[anchor=north]{$\tau$};
\filldraw (0.9091,0) circle (0pt) node[anchor=north]{$\tau^2$};
\filldraw (-0.9091,0) circle (0pt) node[anchor=north]{$\tau^2$};
\draw(0.9091,0.2)--(0.7182,0.8);
\draw(0.5091,0.5)--(0.7182,0.8);
\draw(-0.9091,0.2)--(-0.7182,0.8);
\draw(-0.5091,0.5)--(-0.7182,0.8);
\end{tikzpicture}
\caption{The two caret types for the Bieri-Strebel group corresponding to $x^2+x-1$, known as $\textbf{F}_\tau$, and the interval partitions they each represent. We consider the long leg at depth $2$ and the short leg at depth $1$. }\label{taucarets}
\end{figure}
\end{center}

The final realisation is an extension of the second. With two different types of carets with which to build trees, it is now possible to build two trees with different carets that represent the same interval partition. As such, when we wish to construct equivalence classes of tree-pair diagrams as we did in \ref{treepair}, we have to include in the equivalence relation the ability to switch between equivalent subtrees. We refer to pairs of equivalent trees built with different carets as caret relations.

\begin{center}
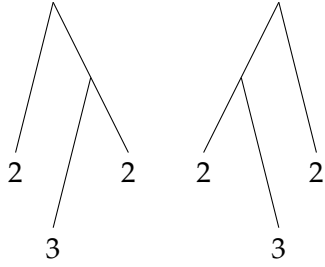
\begin{figure}[H]
\centering
\begin{tikzpicture}[scale=1]
\draw(-1.5,0)--(-2,-2);
\draw(-1.5,0)--(-1,-1);
\draw(-1,-1)--(-1.5,-3);
\draw(-1,-1)--(-0.5,-2);
\draw(1.5,0)--(2,-2);
\draw(1.5,0)--(1,-1);
\draw(1,-1)--(1.5,-3);
\draw(1,-1)--(0.5,-2);
\filldraw(-2,-2) circle (0pt) node[anchor=north]{$2$};
\filldraw(-1.5,-3) circle (0pt) node[anchor=north]{$3$};
\filldraw(-0.5,-2) circle (0pt) node[anchor=north]{$2$};
\filldraw(2,-2) circle (0pt) node[anchor=north]{$2$};
\filldraw(1.5,-3) circle (0pt) node[anchor=north]{$3$};
\filldraw(0.5,-2) circle (0pt) node[anchor=north]{$2$};
\end{tikzpicture}
\caption{Two equivalent trees built from carets in the \FN{\tau} tree-pair presentation.} \label{equitree}
\end{figure}
\end{center}

Burillo, Nucinkis and Reeves are able to use this tree-pair representation to construct an infinite presentation for \FN{\tau}. The presentation contains two different infinite families of generators, corresponding to the two different caret types in the tree-pair representation. The presentation also has a similar set of relations to the $\TF$ presentation in \ref{Fpres}, and also has a set of relations corresponding to the tree equivalence depicted in \ref{equitree}. Combining all of this, the presentation constructed in \cite{Ftau} can be written as

\begin{equation}\label{taupres}
\textsf{F}_{\tau} = \langle x_0, y_0, x_1, y_1,... | a_j b_i = b_i a_{j+1} \textsf{ for } a, b \in \{x,y\}, i<j;x_i x_{i+1} = y_i^2 \rangle
\end{equation}
\par
When considering $x_i$ and $y_i$ as PL-homeomorphisms, we may write them in the following way

\begin{equation}\label{generators} 
\begin{aligned}
x_i(n) = & {\left\{
	\begin{array}{llll}
	n & \mbox{for } 0 \leq n \leq 1-\tau^i, \\
	\tau^{-2} n - \tau^{-1}(1-\tau^i)& \mbox{for } 1-\tau^i \leq n \leq 1-\tau^i+\tau^{i+4}, \\
	n + \tau^{i+3} & \mbox{for } 1-\tau^i+\tau^{i+4} \leq n \leq 1-\tau^{i+1}, \\
	\tau n + \tau^2  & \mbox{for } 1-\tau^{i+1} \leq n \leq 1,
	\end{array}
	\right.} \\
y_i(n) = & {\left\{
	\begin{array}{llll}
	n & \mbox{for } 0 \leq n \leq 1-\tau^i, \\
	\tau^{-1} n - \tau^{-1}(1-\tau^i)& \mbox{for } 1-\tau^i \leq n \leq 1-\tau^{i+1}, \\
	\tau n + \tau^2  & \mbox{for } 1-\tau^{i+1} \leq n \leq 1.
	\end{array}
	\right.}
\end{aligned}
\end{equation}
\\

Burillo, Nucinkis and Reeves were also able to reduce this infinite presentation down to a finite presentation containing $4$ generators and $10$ relations. As with the finite presentation for $\TF$, this presentation is cumbersome to use in comparison to the infinite presentation, but can be seen at (\cite{Ftau}, Section 4).
\par
In his thesis \cite{Winstone}, Winstone was able to generalise these results to the quadratic Bieri-Strebel groups, that is to say the algebraic Bieri-Strebel groups with a quadratic associated subdivision polynomial. In doing so, Winstone developed two theorems relevant to our discussion, as well as an infinite presentation for a subset of these groups. Quite possibly the most important result from \cite{Winstone} is the following:

\begin{citing}\label{deftree}
(\cite{Winstone}, Theorem 1.2.3) For a quadratic Bieri-Strebel group \FN{\beta} with subdivision polynomial of the form $ax^2+bx-1$, \FN{\beta} has a well defined tree-pair representation if and only if $a \leq b$.
\end{citing}

When considering tree-pair representations of other quadratic Bieri-Strebel groups, we have to consider the rapid increase of caret types as the coefficients of the subdivision polynomial increase. As depicted in \ref{taucarets}, the subdivision polynomial dictates the possible interval partitions, and therefore the possible caret types. The important observation is that the number of intervals in a partition for \FN{\beta} is equal to $a+b$, where \FN{\beta} has the subdivision polynomial $ax^2+bx-1$. We would expect any caret to have $a$ legs of length $2$ and $b$ legs of length $1$ (to correspond to the $a$ intervals of length $\beta^2$ and the $b$ intervals of length $\beta$ in a partition for \FN{\beta}. As such, there are ${a+b}\choose{a}$ possible caret types in the tree pair representation for \FN{\beta}.

\begin{center}
\begin{figure}[H]
\centering
\begin{tikzpicture}[scale=2]
\draw (0,0) -- (-1,-2);
\draw (0,0) -- (0,-1);
\draw (0,0) -- (0.5,-1);

\draw(2,0) -- (1.5,-1);
\draw(2,0) -- (2,-2);
\draw(2,0) -- (2.5,-1);

\draw(4,0) -- (3.5,-1);
\draw(4,0) -- (4,-1);
\draw(4,0) -- (5,-2);
\end{tikzpicture}
\caption{The three possible caret types in \FN{\beta} with subdivision polynomial $x^2+2x-1$.} 
\end{figure}
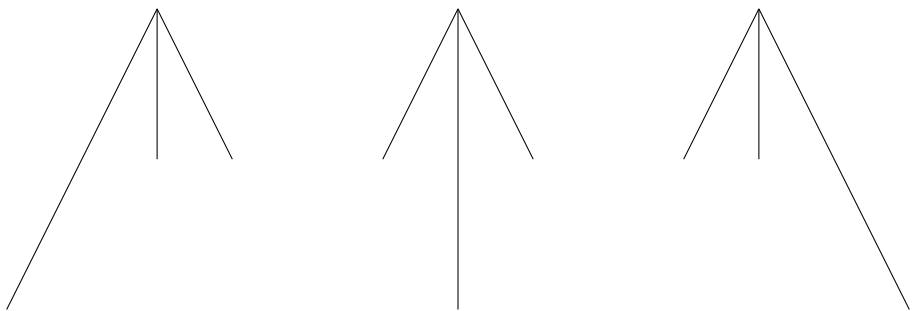
\end{center}

Winstone's second major result simplifies the large variety of caret types for quadratic Bieri-Strebel groups, making tree-pair representations much more manageable.

\begin{citing}\label{twocaret}
(\cite{Winstone}, Remark 35) There exist two caret types in each well-defined tree pair representation of a quadratic Bieri-Strebel group such that for any tree within that tree-pair representation, an equivalent tree may be constructed using only carets of those two types
\end{citing}

This result is not just important for simplifying the tree-pair representations, but also leads into the general presentation for quadratic Bieri-Strebel groups that have tree-pair representations. Reducing each tree pair representation once again allows us to generate \FN{\beta} from two infinite families of generators. Using these caret types, and the caret relations between them, Winstone was able to construct infinite presentations for a large number of quadratic Bieri-Strebel groups.

\begin{citing}\label{betapres}
(\cite{Winstone}, Theorem 1.2.5) Let \FN{\beta} be an algebraic Bieri-Strebel group with associated subdivision polynomial $ax^2+bx-1$. If $a \leq b$, then \FN{\beta} has the infinite presentation:

\begin{equation}
\textbf{F}_\beta = \langle x_0, y_0, x_1, y_1,...| R_1, R_2 \rangle
\end{equation}

where $R_1$ and $R_2$ are the relations:

\begin{equation}
\begin{aligned}
R_1 & : f_j g_i = g_i f_{j+a+b-1} \textsf{ for } g,f \in \{x,y\}, i < j \\
R_2 & : x_{i+a} x_{i+a+1}...x_{i+2a-1} x_i = y_{i} y_{i+1} ... y_{i+a-1} y_i \textsf{ for all } i \geq 0
\end{aligned}
\end{equation}
\end{citing}

If we so desired, we could reduce these infinite presentations to finite presentations. We have no immediate use for the finite presentations and so will leave their existence implied rather than explicitly writing them out.

For quadratic Bieri-Strebel groups with subdivision polynomial $ax^2+bx-1$, $a>b$, not much is known outside of generalities for all Bieri-Strebel groups or other large subsets of such groups. Winstone (\cite{Winstone}, Theorem 1.2.4) concludes that each such group has an element $g$ such that $g$ cannot be represented either as a partition pair or as a tree pair. Owen Tanner (\cite{Tanner}, Theorem 3) was able to show that all such groups are finitely generated.
\par
Of particular interest are the groups with polynomial of the form $(n-1)nx^2+x-1$ for $n \geq 2$. This polynomial can be factorised as $(nx-1)((n-1)x+1)$ and therefore has the unique positive root $\frac{1}{n}$. By the definition of Algebraic Bieri-Strebel group given in \ref{algdef}, we would expect the group with this subdivision polynomial to have slopes in $\langle \frac{1}{n} \rangle$ and breakpoints in $\mathbb{Z}[\frac{1}{n}]$, which would make it the same as \FN{n}. Notably, \FN{n} can be expressed as the algebraic Bieri-Strebel group with subdivision polynomial $nx-1$, which when used to compute interval divisions and tree pairs in the style of Winstone produces the tree-pairs described in \ref{fn}. This leads us to make the following conjecture:

\begin{conjecture}
A well defined tree-pair for an algebraic Bieri-Strebel group \FN{\beta} can only be derived from an associated subdivision polynomial $P(x)$ if $P(x)$ is the minimal polynomial with the root $\beta$ among polynomials of the form $a_nx^n + ... + a_1x -1$, $a_i \in \mathbb{Z}_{\geq 0}$.
\end{conjecture}

Among linear and quadratic subdivision polynomials, we can see evidence for this conjecture. In particular, we know that all linear subdivision polynomials take the form $nx-1$ for some $x$ (with the polynomial having root $\frac{1}{n}$. Should a quadratic subdivision polynomial have the root $\frac{1}{n}$, then by the factor theorem, that polynomial must have $(x-\frac{1}{n})$ as a factor, which can be expressed as $nx-1$ when working with integer polynomials. We can now use what we know about the general form of subdivision polynomials to determine facts about other factors of the quadratic polynomial. Since the polynomial is quadractic, it can only be factorised as the product of two linear polynomials, meaning we can write

\begin{equation}
ax^2+bx-1 = (nx-1)(cx+d)
\end{equation}

Immediately, we can conclude that $d=1$, as $-1*1=-1$. Similarly, we can conclude $c>0$, as otherwise $a \leq 0$, which would mean the quadratic polynomial would not be a subdivision polynomial. That leaves us with the two equations $a=nc$ and $b=n-c$. As $n$ and $c$ are both positive integers, we can therefore conclude that $a>b$. As such, any time a quadratic subdivision polynomial $ax^2+bx-1$ shares its unique positive root with a linear subdivision polynomial $nx-1$, we have that $a>b$ and it therefore falls outside the set of subdivision polynomials for which Winstone derived tree pairs.

We note that this statement cannot be an if and only if. That is to say that it is not necessarily the case that a subdivision polynomial that does not share a root with a sudivision polynomial of lesser degree will have a tree pair representation. As a brief example, the polynomial $3x^2+x-1$ has the root $\frac{\sqrt{13}-1}{6}$, which is clearly irrational and therefore cannot be the root of any linear polynomial, but by \ref{deftree}, this subdivision polynomial does not have a well defined tree pair representation. 

\section{Higher Order Polynomials}

For algebraic Bieri-Strebel groups with higher order subdivision polynomials, little is known outside of general results for Bieri-Strebel groups, such as Tanner's finite generation result \cite{Tanner}. When working with higher order polynomials, we expect a generalisation of the methods used to construct tree-pair representations for quadratic Bieri-Strebel groups. In particular, each caret in a tree=pair representation for the group with subdivision polynomial $ax^2+bx-1$ has $a$ legs of length $2$ and $b$ legs of length $1$. Therefore, when generalising to higher order polynomials, we would expect the group with subdivision polynomial $\Sigma_{i=1}^n a_i x^i -1$ will have $a_i$ legs of length $i$. Following this generalisation, a caret relation is known for the group corresponding to the polynomial $x^3+x-1$, but it is unknown if this allows a well-defined tree-pair representation.

\begin{center}
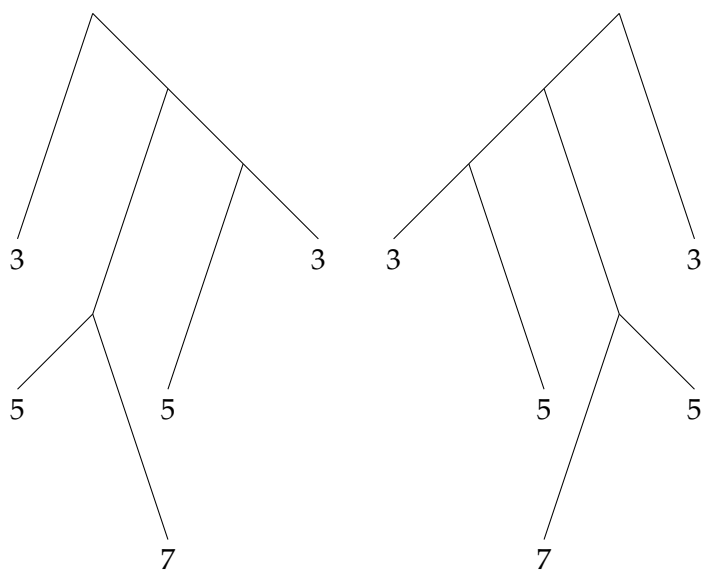
\begin{figure}[H]
\centering
\begin{tikzpicture}[scale=1]
\draw (0,0) -- (-1,-3);
\draw (0,0) -- (1,-1);
\draw (1,-1) -- (0,-4);
\draw (1,-1) -- (2,-2);
\draw (2,-2) -- (3,-3);
\draw (2,-2) -- (1,-5);
\draw (0,-4) -- (-1,-5);
\draw (0,-4) -- (1, -7);
\filldraw (-1,-3) circle (0pt) node[anchor=north]{$3$};
\filldraw (-1,-5) circle (0pt) node[anchor=north]{$5$};
\filldraw (1,-7) circle (0pt) node[anchor=north]{$7$};
\filldraw (1,-5) circle (0pt) node[anchor=north]{$5$};
\filldraw (3,-3) circle (0pt) node[anchor=north]{$3$};

\draw (7,0) -- (8,-3);
\draw (7,0) -- (6,-1);
\draw (6,-1) -- (7,-4);
\draw (6,-1) -- (5,-2);
\draw (5,-2) -- (4,-3);
\draw (5,-2) -- (6,-5);
\draw (7,-4) -- (8,-5);
\draw (7,-4) -- (6, -7);
\filldraw (8,-3) circle (0pt) node[anchor=north]{$3$};
\filldraw (8,-5) circle (0pt) node[anchor=north]{$5$};
\filldraw (6,-7) circle (0pt) node[anchor=north]{$7$};
\filldraw (6,-5) circle (0pt) node[anchor=north]{$5$};
\filldraw (4,-3) circle (0pt) node[anchor=north]{$3$};
\end{tikzpicture}
\caption{The caret relation for $x^3+x-1$. It differs from quadratic caret relations in that each tree uses both caret types.} 
\end{figure}
\end{center}

While determining well-defined tree pairs for higher order Bieri-Strebel groups is an as-yet unsolved problem, we are able to determine some cases where a tree-pair representation based on the subdivision polynomial does not work. 

\begin{theorem}\label{roothm}
The Bieri-Strebel group with subdivision polynomial $ax^{2n}+bx^n-1$ cannot have a well defined tree-pair representation.
\end{theorem}

We begin our proof with the following example.

\begin{example} \label{rootfind}
For all $0<p \in \mathbb{Z}[\beta]$, where $\beta$ is the root of an $n$-th degree polynomial, is it possible to write:

$$
p = b_0 + b_1 \beta +...+b_{n_1} \beta^{n-1}
$$

where $b_i \in \mathbb{Z}_{\geq 0}$ for all $i$?
\end{example}

The following counterexample works as a counterexample to Winstone's theorem (\cite{Winstone}, Theorem 2.3.3).

\begin{proof}[Counterexample]
Consider the polynomial $x^4+x^2-1$, with the root $\sqrt{\tau}$. We will take $p=1-\sqrt{\tau}$. We work with the reciprocal of our root, taking $\beta = \sqrt{\tau}^{-1}$, which means $p=1-\beta^{-1}$, and as we have that $\beta^4-\beta^2-1=0$, we can say that $\beta^{-1}=\beta^3 -\beta$, and therefore $p=1-\beta^3+\beta$.
\par
We can represent the action of substituting $\beta^3 = \beta^{-1} - \beta$ into $p$  by writing the coefficients of $p$ as a vector $\underline{p}$ and expressing the substitution as the matrix $A$ based on the polynomial $x^4-x^2-1$, rewritten as $\beta^4 = \beta^2 + 1$
$$
A =
\begin{pmatrix}
0 & 1 & 0 & 0 \\
1 & 0 & 1 & 0 \\
0 & 0 & 0 & 1 \\
1 & 0 & 0 & 0
\end{pmatrix}
; \underline{p} = 
\begin{pmatrix}
-1 \\ 0 \\ 1 \\ 1
\end{pmatrix}
$$

To simplify a later segment, we will calculate 

$$A \underline{p} = \begin{pmatrix} 0 \\ 0 \\ 1 \\ -1 \end{pmatrix}$$

Were it possible to express $p$ with only positive coefficients of powers of $\beta$, there would be $N$ such that $A^N \underline{p}$ has only positive entries. We will use the constructed example to show this is not the case. First, we will calculate 

$$
A^4 =
\begin{pmatrix}
2 & 0 & 1 & 0 \\
0 & 2 & 0 & 1 \\
1 & 0 & 1 & 0 \\
0 & 1 & 0 & 1 
\end{pmatrix}
$$

Here we wish to emphasise the structure of the matrix, with alternating entries in each row and column being empty. We will consider a general matrix of this form and multiply it by $A$:

$$
\begin{pmatrix}
a & 0 & b & 0 \\
0 & c & 0 & d \\
e & 0 & f & 0 \\
0 & g & 0 & h 
\end{pmatrix}
\begin{pmatrix}
0 & 1 & 0 & 0 \\
1 & 0 & 1 & 0 \\
0 & 0 & 0 & 1 \\
1 & 0 & 0 & 0
\end{pmatrix}
=
\begin{pmatrix}
0 & a & 0 & b \\
c+d & 0 & c & 0 \\
0 & e & 0 & f \\
g+h & 0 & g & 0 
\end{pmatrix}
$$

Which has a similar alternating pattern, but all the nonzero entries are now zero and all the zero entries are now nonzero. Multiplying this matrix by $A$ again gives us. 

$$
\begin{pmatrix}
0 & a & 0 & b \\
c+d & 0 & c & 0 \\
0 & e & 0 & f \\
g+h & 0 & g & 0 
\end{pmatrix}
\begin{pmatrix}
0 & 1 & 0 & 0 \\
1 & 0 & 1 & 0 \\
0 & 0 & 0 & 1 \\
1 & 0 & 0 & 0
\end{pmatrix}
=
\begin{pmatrix}
a+b & 0 & a & 0 \\
0 & c+d & 0 & c \\
e+f & 0 & e & 0 \\
0 & g+h & 0 & g 
\end{pmatrix}
$$

which is of the same structure as $A^4$. As such, we can see that $A^{2N}$ will have the same structure as $A^4$ for $N \geq 2$, an $A^{2N+1}$ will have the same structure as $A^5$, for all $N \geq 2$. We can now show that neither of these structures can create a vector $A^N \underline{p}$ such that all entries in the vector are nonnegative. We can simplify by using $A \underline{p}$ as the vector, from which we can clearly see:

$$
\begin{pmatrix}
a & 0 & b & 0 \\
0 & c & 0 & d \\
e & 0 & f & 0 \\
0 & g & 0 & h 
\end{pmatrix}
\begin{pmatrix} 0 \\ 0 \\ 1 \\ -1 \end{pmatrix}
=
\begin{pmatrix} b \\ -d \\ f \\ -h \end{pmatrix}  
$$

$$
\begin{pmatrix}
0 & a & 0 & b \\
c & 0 & d & 0 \\
0 & e & 0 & f \\
g & 0 & h & 0 
\end{pmatrix} 
\begin{pmatrix} 0 \\ 0 \\ 1 \\ -1 \end{pmatrix}
=
\begin{pmatrix} -b \\ d \\ -f \\ h \end{pmatrix}
$$
  
As $A$ is a nonnegative matrix, all entries in $A^N$ will be nonnegative for all $N$, and so $A^N \underline{p}$ will always contain two negative numbers for $N \geq 5$. 
\end{proof}

It is straightforward to see that the alternating nature of this counterexample will generalise to any polynomial of the form $ax^4+bx^2-1$, as the resulting polynomial $A$ will just be 

$$
\begin{pmatrix}
0 & 1 & 0 & 0 \\
a & 0 & 1 & 0 \\
0 & 0 & 0 & 1 \\
b & 0 & 0 & 0
\end{pmatrix}
$$

and as such we have

$$
A^4 = 
\begin{pmatrix}
a+b^2 & 0 & b & 0 \\
0 & a+b^2 & 0 & b \\
ab & 0 & a & 0 \\
0 & ab & 0 & a
\end{pmatrix}
$$

which has the same structure as $A^4$ from the counterexample. As such, the alternating patten will repeat. We can then construct a $\underline{p}$ such that $A^N\underline{p}$ always has a negative number.
\par
What is less immediately obvious is that this pattern extends to any $a_n x^n + ... + a_1 x -1$ with the condition that the set of $\{i_1,...,i_k\}$ where $a_{i_j}$ is nonzero is not coprime. In this case, we have that $A^N$ has the same structure as $A^N+k$, where $k = \textsf{hcf}(i_1,...,i_k)$. However, this can be seen with elementary matrix multiplication.
\par
The next step is to show the following:

\begin{lemma}\label{rootexclude}
For any $\beta$ as the root of a quadratic subdivision polynomial $ax^2+bx-1$, $\sqrt{\beta}$ is not in $\mathbb{Z}[\beta]$.
\end{lemma}
\begin{proof}
By definition,  elements of $\mathbb{Z}[\beta]$ have the form $c_0 + c_1 \beta + c_2 \beta^2 + \dots$. However, because $\beta$ is the root of the quadratic polynomial $ax^2+bx-1$, we can write $\beta^2 = \frac{1}{a} - \frac{b}{a} \beta$. As such, any term with a $\beta^2$ or higher power can be substituted, allowing us to express any element in $\mathbb{Z}$ in the form $c+d \beta$.

Now, suppose $\sqrt{\beta} \in \mathbb{Z}[\beta]$. We can then write $\sqrt{\beta} = c+d \beta$. From the property of the square root, we can then write

\begin{equation}
\begin{aligned}
(c+ d \beta)(c+d \beta) & = \beta \\
c^2 +2cd \beta + d^2 \beta^2  & = \beta \\
c^2 +2cd \beta +d^2 (\frac{1}{a}-\frac{b}{a} \beta) & = \beta \\
c^2 + \frac{d^2}{a} + (2cd - \frac{d^2 b}{a}) \beta & = \beta 
\end{aligned}
\end{equation}

As $\beta$ is irrational, we know that the only way for this equation to be true is if the constant coefficients on each side are equal, and so are the $\beta$ coefficients. This leaves us with the simultaneous equations

\begin{equation}
\begin{aligned}
c^2 + \frac{d^2}{a} & = 0 \\
2cd - \frac{d^2 b}{a} & = 1 \\
\end{aligned}
\end{equation}

From the first equation $c^2$ and $d^2$ are squares of numbers in $\mathbb{Z}$, and therefore must be positive, and $a$ is positive by the definition of a subdivision polynomial (as the subdivision polynomial in question is quadratic, $a \neq 0$ and so the equation is well defined). As such, both terms of the sum must be nonnegative and so it can only hold if $c=d=0$. However, setting $c=d=0$ in the second equation results in the equation reducing to $0=1$. As such we have a contradiction and so our assumption that $\sqrt{\beta} \in \mathbb{Z}$ must be incorrect.
\end{proof}

We can also add this brief corollary:

\begin{corollary} \label{rootcor}
For any $\beta$ as the root of a quadratic subdivision polynomial $ax^2+bx-1$, $n \in \mathbb{N}$, $\sqrt[2n]{\beta}$ is not in $\mathbb{Z}[\beta]$.
\end{corollary}
\begin{proof}
$\mathbb{Z}[\beta]$ is closed multiplicatively, so if $\sqrt[2n]{\beta}$ was in $\mathbb{Z}[\beta]$, that would imply $(\sqrt[2n]{\beta})^n = \sqrt{\beta}$ is in $\mathbb{Z}[\beta]$, which we know is not true by \ref{rootexclude}.
\end{proof}

Consider the tree pair representation formed from the non-coprime power subdivision polynomial $P_1=P(x)=a_n x^{kn} + a_{n-1} x^{k(n-1)} +...+ a_1 x^k -1$, where $k \in \mathbb{N}$. Based on \ref{wintree}, we would expect carets in this tree-pair representation to have $a_n$ legs of length $k_n$, $a_{n-1}$ legs of length $k(n-1)$ and so on. We now consider the polynomial $P_2=P(\sqrt[k]{x}) = a_n x^n + a_{n-1} x^{n-1} + ... + a_1 x -1$. The carets in the tree-pair representation for this polynomial will have $a_n$ legs of length $n$, $a_{n-1}$ legs of length $n-1$ and so on. From here, we can consider the map $i$ from the set of carets for $P(x)$ to the set of carets for $P(\sqrt[k]{x})$, which maps each caret to the caret obtained by dividing the length of each leg by $k$. We can determine the number of possible carets in the caret set for $P(x)$ by determining the possible distinct orderings for the set of legs each caret has, which is $a_n$ legs of length $kn$, $a_{n-1}$ legs of $k_{n-1}$, and so on. Similarly, the set of possible carets for $P(\sqrt[k]{x})$ is determined by the possible orderings of $a_n$ carets of length $n$, $a_n-1$ legs of length $n-1$ etc. As such, the set of legs being ordered for $P(x)$ consists of the same number of subsets $S[P_1]_i$ containing legs of the same length as the set of legs being ordered for $P(\sqrt[k]{x})$, and these subsets can be ordered in such a way that $|S[P_1]_i|=|S[P_2]_i|$ for all $i$. Thus the number of possible orderings for legs on the $P(x)$ carets is the same as the number of possible orderings for the $P(\sqrt[k]{x})$, so the map is surjective, and if two carets in $P(x)$ are mapped to the same caret in $P(\sqrt[k]{x})$, then they must have the same ordering of legs, and are thus the same caret, and so the map is injective. 
\par
We now induce a map $i^*$  to tree pairs in the tree pair representation of $P(x)$ by applying $i$ to each caret in the tree pair individually, preserving adjacency between carets. The path to each leaf from the root of the tree can be broken down into legs for carets, and if $i^*$ applied to tree pairs just applies the map $i$ to each caret in the tree, then the depth of each leaf will be mapped from $n$ to $\frac{n}{k}$.
\par
We consider the path from the root of a tree to an arbitrary leaf in a tree pair diagram in the Bieri-Strebel group with $P(x)$ as subdivision polynomial. each leg in this path is of length $k*b$ for some $1 \leq b \leq n$. Thus the full path can be expressed as $\Sigma_{i=1}^m k*b_i$ where $m$ is the number of legs in the path. When we apply $i*$ to this tree pair, each leg in this path is mapped from a leg of length $k*b_i$ to a leg of length $b_i$. Thus the length of this path is just $\Sigma_{i=1}^m b_i$. As each $b_i$ is a natural number, so is $\Sigma_{i=1}^m b_i$, thus the length of the path from the root to the leaf is a natural number, so the depth of each leaf must be a natural number.
\par
Consider now the preimage under $i^*$ of a tree pair $[T_1,T_2]$ in the space of tree pair diagrams for $P(\sqrt[k]{x})$. The depth of the $i$th leaf in $[T_1,T_2]$ has depth $d_i$ and so any tree pair $[T'_1,T'_2]$ in the preimage of $[T_1,T_2]$ under $i$ must be such that the depth of the $i$th leaf is of depth $k*d_i$. The equivalence relation between tree pair relations (in particular the caret relations) dictate that two tree pairs are in the same equivalence class if they have the same leaf depths in the same order, and so the map $i^*$ must map equivalence classes to equivalence classes and cannot map two different equivalence classes into the same equivalence class. We can see that it is surjective by considering the map $i_*$, that multiplies each leaf depth by $k$. Clearly $i^* i_*$ is the identity map on the tree-pair representation for $P(\sqrt[k]{x})$, so  for each $[T,T]$ in the tree pair representation for $P(\sqrt[k]{x})$ there must be a tree pair $[T',T']$ in the tree pair representation for $P(x)$ such that $i_*([T,T])=[T',T']$, $i^*([T',T'])= [T,T]$, and hence $i^*$ is surjective. 

We can also see that $i^*$ is a homomorphism with regard to tree pair composition. We can see this by performing tree-pair composition simultaneously in $P(x)$ and $P(\sqrt[k]{x})$. Whenever we add a caret $c$ while working in $P(x)$, we add $i(c)$ to the same leaf in $P(\sqrt[k]{x})$. Thus $i^*$ is an isomorphism between the equivalence classes in the tree-pair representations for $P(x)$ and $P(\sqrt[k]{x})$. This leads us to the following lemma:

\begin{lemma}\label{treepairdrop}
Consider $\beta$ as the unique positive root of $P(x) = a_n x^n +...+ a_1 x -1$ and $\sqrt[k]{\beta}$ as the root of $P(x^k) = a_n x^{kn} + ... + a_1 x^k -1$. Suppose the tree pair representation based of $P(x)$ is a well-defined tree pair representation for \FN{\beta}. Either \FN{\sqrt[k]{\beta}} $\cong$ \FN{\beta}, or the tree-pair based on $P(x^k)$ is not a well-defined tree-pair representation for \FN{\sqrt[k]{\beta}}. 
\end{lemma}

\begin{proof}[Proof of theorem \ref{roothm}]
To prove \ref{roothm}, we wish to construct an element we know is in \FN{\sqrt[2n]{\beta}} but cannot be in \FN{\beta}. We will construct an element that is in \FN{\sqrt{\beta}}, and by a similar argument to the proof of \ref{rootcor}, must be in \FN{\sqrt[2n]{\beta}} as well.
\par
We first create a partition pair which we know cannot be in \FN{\beta}. From \ref{rootexclude}, we know that if a partition pair contains $\sqrt{\beta}$ as a nontrivial breakpoint, then that partition pair cannot be an element of \FN{\beta}. From here, the simplest solution would be to have $\sqrt{\beta}$ as the sole breakpoint in the first partition and then have $1-\sqrt{\beta}$ as the sole breakpoint of the second partition. The issue with this is that the resulting map would map an interval of length $\sqrt{\beta}$ to an interval of length $1-\sqrt{\beta}$. This would require $\frac{1-\sqrt{\beta}}{\sqrt{\beta}}$ to be in $\langle \sqrt{\beta} \rangle$, which is challenging to prove. Instead, we will construct a partition pair that has only slopes in $\langle \beta \rangle$. Since $\langle \beta \rangle \subseteq \langle \sqrt{\beta} \rangle$, we can then guarantee our partition pair is in \FN{\sqrt{\beta}}.
\par
We create our partition pair in the following way. First we take two copies of the interval with $\sqrt{\beta}$ as the sole breakpoint. From here, we subdivide the interval of length $\sqrt{\beta}$ in each partition in different ways. Each interval will now be divided into $a$ intervals of length $\beta^2 * \sqrt{\beta}$ and $b$ intervals of length $\beta * \sqrt{\beta}$, where $\beta$ is the root of $ax^2+bx-1$, implying 

\begin{equation}
\begin{aligned}
a \beta^2 + b \beta & = 1 \\
a \beta^2 *\sqrt{\beta} + b \beta * \sqrt{\beta} & = \sqrt{\beta}
\end{aligned}
\end{equation}
\par
The ordering of the subintervals in the partitions is arbitrary, aside from 2 points. First, the sole interval of length $1-\sqrt{\beta}$ must be the $i$th interval in each partition for some $i$, and the partitions must not be identical. 

\begin{center}
\begin{figure}[H]
\centering
\begin{tikzpicture}[scale=8]
    \draw[thick] (1,0) -- (0,0);
    \draw[thick] (0,-0.2) -- (1, -0.2);
    \draw[thick] (0, 0.03) -- (0, -0.03);
    \draw[thick] (1,-0.03) -- (1, 0.03);
    \draw[thick] (0, -0.23) -- (0,-0.17);
    \draw[thick] (1, -0.23) -- (1, -0.17);
    \draw[thick] (0.8, 0.02) -- (0.8, -0.02);
    \draw[thick] (0.8, -0.22) -- (0.8, -0.18);
    \draw[thick] (0.5, 0.02) -- (0.5, -0.02);
    \draw[thick] (0.3, -0.22) -- (0.3, -0.18);
    \filldraw (0,0.03) circle (0pt) node[anchor=south]{$0$};
    \filldraw (1,0.03) circle (0pt) node[anchor=south]{$1$};
    \filldraw (0,-0.23) circle (0pt) node[anchor=north]{$0$};
    \filldraw (1,-0.23) circle (0pt) node[anchor=north]{$1$};
    \filldraw (0.8,-0.23) circle (0pt) node[anchor=north]{$\sqrt{\tau}$};
    \filldraw (0.8,0.03) circle (0pt) node[anchor=south]{$\sqrt{\tau}$};
    \filldraw (0.5,0.03) circle (0pt) node[anchor=south]{$\tau^{\frac{3}{2}}$};
    \filldraw (0.3,-0.23) circle (0pt) node[anchor=north]{$\tau^{\frac{5}{2}}$};
    \node (A) at (0.25,0) {};
    \node (B) at (0.65,0) {};
    \node (C) at (0.9,0) {};
    \node (D) at (0.15,-0.2) {};
    \node (E) at (0.55,-0.2) {};
    \node (F) at (0.9,-0.2) {};
    \draw[->](A)--(D);
    \draw[->](B)--(E);
    \draw[->](C)--(F);

\end{tikzpicture}
    
\caption{An example of a partition pair in \FN{\sqrt{\tau}} with the breakpoint $\sqrt{\tau}$ excluding it from \FN{\tau}. The slopes in this map have gradients $\tau$, $\tau^{-1}$ and $1$ from left to right}\label{roottau}
\end{figure}
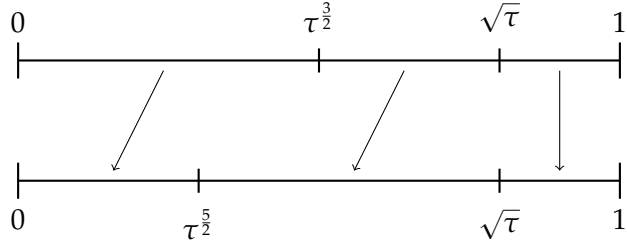
\end{center}

As seen in the example \ref{roottau}, maps constructed in this way have gradients  in $\langle \sqrt{\beta} \rangle$ (indeed, they are in $\langle \beta \rangle$), but as they have $\sqrt{\beta}$ as a breakpoint, these maps are in \FN{\sqrt{\beta}} but not in \FN{\beta}. Thus, we can conclude that \FN{\beta} is not the same group as \FN{\sqrt{\beta}} and so the tree pair representation based on $P(x^2)$ is not a well-defined tree pair representation for \FN{\sqrt{\beta}} by \ref{treepairdrop}. Similarly, this element is in \FN{\sqrt[2n]{\beta}}, and so $P(x^{2n})$ is not a well-defined tree pair representation for \FN{\sqrt[2n]{\beta}}.
\par
Finally, we need to preclude the possibility of a different tree pair representation. For this we rely on our example to \ref{rootfind}. What we have shown with the example (and its generalisation) is that we can form an interval in a partition in \FN{\sqrt{\beta}} with a length that cannot be expressed in the form $a_4 \beta^{k+4} + a_3 \beta{k+3} + a_2 \beta^{k+2} + a_1 \beta^{k+1}$ via repeated subdivision (the substitution that multiplication by the matrix $A$ represents). However, were we constructing a tree to represent this partition, such an interval would have to be expressed in this way. As such, these intervals cannot be expressed as part of a tree pair presentation and thus we have our proof that Bieri-Strebel groups with associated polynomial $ax^4+bx^2-1$ cannot have a well defined tree-pair representation. Combining this with \ref{rootcor} is our proof of \ref{roothm}.

\end{proof}
\par
While this seems challenging to generalise in its entirety, we would like to offer this conjecture regarding non-coprime power polynomials:

\begin{conjecture}
The Bieri-Strebel group with associated subdivision polynomial $a_{n}x^{kn} + a_{n-1}x^{k(n-1)}+...+a_1x^k -1, k>1$ with not all $a_i=0$ does not have a well-defined tree-pair representation.
\end{conjecture} 

\printbibliography[heading=bibintoc]

\end{document}